\documentclass{amsart}
\usepackage[utf8]{inputenc}
\usepackage{setspace}
\usepackage{mathrsfs}
\usepackage{amssymb}
\usepackage{latexsym}
\usepackage{amsfonts}
\usepackage{amsmath}
\usepackage{eucal}
\usepackage{bm}
\usepackage{bbm}
\usepackage{graphicx}
\usepackage[english]{varioref}
\usepackage[nice]{nicefrac}
\usepackage[all]{xy}
\usepackage{amsthm}
\usepackage{amssymb,amsthm,upref,amscd}
\usepackage{tikz}

\def\op{\operatorname}

\def\mmod{\kern-1pt\operatorname{-mod}}

\newtheorem{Thm}{Theorem}[section]
\newtheorem{Lem}[Thm]{Lemma}
\newtheorem{Cor}[Thm]{Corollary}
\newtheorem{Prop}[Thm]{Proposition}
\newtheorem{Def}[Thm]{Definition}
\newtheorem{Rem}[Thm]{Remark}

\newtheorem{Conj}[Thm]{Conjecture}

\numberwithin{equation}{section}

\begin{document}

\title[The Principal Representations of Reductive Groups ]{The Principal Representations of Reductive Algebraic Groups with Frobenius Maps}

\author{Junbin Dong}
\address{Institute of Mathematical Sciences, ShanghaiTech University, 393 Middle Huaxia Road, Pudong, Shanghai 201210, PR China.}
\email{dongjunbin1990@126.com}



\date{January 6, 2021}


\begin{abstract}
We introduce the principal representation category $\mathscr{O}({\bf G})$  of reductive algebraic groups with Frobenius maps and put forward a conjecture that this category is a highest weight category.
When $\Bbbk$ is complex field $\mathbb{C}$, we provide some evidences of this conjecture. We also study certain kind of bound quiver algebras whose representations are related to the principal representation category  $\mathscr{O}({\bf G})$ .
\end{abstract}

\maketitle

\section*{Introduction}
Let ${\bf G}$ be a connected reductive algebraic  group defined over the finite field $\mathbb{F}_q$ with the standard Forbenius map $F$ induced by the automorphism $x\mapsto x^q$ on $\bar{\mathbb{F}}_q$. Let  $\Bbbk$ be a field. According to  a result of Borel and Tits \cite[Theorem 10.3 and Corollary 10.4]{BT}, we know that except the trivial representation, all other irreducible representations of $\Bbbk {\bf G}$ are infinite-dimensional if ${\bf G}$ is a semisimple algebraic group over $\bar{\mathbb{F}}_q$ and $\Bbbk $ is infinite with $\op{char}\Bbbk\neq \op{char} \bar{\mathbb{F}}_q$. So it seems to be difficult to study the abstract representations of ${\bf G}$. However in \cite{X}, Nanhua Xi studied the abstract representations of ${\bf G}$ over ${\bf \Bbbk}$ by taking the direct limit of the finite-dimensional representations of $G_{q^a}$ and got many interesting results.
Later, motivated by Xi's idea, the structure of the permutation module $\Bbbk [{\bf G}/{\bf B}]$ (${\bf B}$ is a Borel subgroup of ${\bf G}$ ) was studied in  \cite{CD1} for the cross characteristic case and \cite{CD2} for the defining characteristic case.
The paper \cite{CD3} studied the general abstract induced module $\mathbb{M}(\theta)=\Bbbk{\bf G}\otimes_{\Bbbk{\bf B}}{\Bbbk}_\theta$ for any field $\Bbbk$ with $\op{char}\Bbbk\neq \op{char} \bar{\mathbb{F}}_q$ or $\Bbbk=\bar{\mathbb{F}}_q$, where $\bf T$ is a maximal splitting torus contained in a $F$-stable Borel subgroup $\bf B$ and $\theta \in \widehat{\bf T}$ (the character group of $\bf T$). The  induced module $\mathbb{M}(\theta)$  has a composition series (of finite length) if $\op{char}\Bbbk\neq \op{char} \bar{\mathbb{F}}_q$. In the case $\Bbbk=\bar{\mathbb{F}}_q$ and $\theta$ is a rational character, $\mathbb{M}(\theta)$ has such composition series if and only if $\theta$ is antidominant (see \cite{CD3} for details). In both cases, the composition factors has the form $E(\theta)_J$ with $J\subset I(\theta)$
(see Section 1 for explicit definition).

The construction of $\mathbb{M}(\theta)$  is similar to the Verma module in the representations of semisimple complex Lie algebras. Thus  motivated by  the famous BGG category $\mathscr{O}$,  this paper introduces a category $\mathscr{O}({\bf G})$ called principal representation category to study the abstract representations of  ${\bf G}$. It is the full subcategory of $\Bbbk{\bf G}$-Mod such that any object $M$ in $\mathscr{O}({\bf G})$ is of finite length and its composition factors are $E(\theta)_J$ for some $\theta \in \widehat{\bf T}$ and $J\subset I(\theta)$.   The BGG category is a highest weight category (see \cite[Example 3.3 (c)]{CPS}).  Thus in the study of category $\mathscr{O}({\bf G})$, we put forward a conjecture that this category  has enough injectives. So it is a highest weight category (in the sense  of \cite{CPS}) when we assume that Conjecture \ref{conjecture} is valid.  By the property of highest weight category,  $\mathscr{O}({\bf G})$ has a decomposition $\mathscr{O}({\bf G})= \displaystyle \bigoplus_{\theta \in \widehat{\bf T}} \mathscr{O}({\bf G})_\theta$, where  $\mathscr{O}({\bf G})_\theta$ is the  full subcategory of $\mathscr{O}({\bf G})$ containing the objects whose subquotients are $E(\theta)_J$ for a fixed character $\theta$ of $\bf T$.
The weight set  of $\mathscr{O}({\bf G})_\theta$ is finite and therefore there exists a  finite-dimensional quasi-hereditary algebra $A_\theta$ such that $\mathscr{O}({\bf G})_\theta$ is equivalent to the right $A_\theta$-modules. Thus all the indecomposable projective objects in  $\mathscr{O}({\bf G})$ are given in Section 3. For each $\theta \in \widehat{\bf T}$, the algebra $A_\theta$ is isomorphic to a bound quiver algebra $\mathscr{A}_n$ (which is defined in Section 4) when $|I(\theta)|=n$.
When $|I(\theta)|=1~\text{or}~2$, this algebra is of finite representation type which means that the number of the indecomposable modules up to isomorphic is finite.
By the equivalence of $\mathscr{O}({\bf G})_\theta$ and the the right $A_\theta$-modules, we give all the indecomposable modules of $\mathscr{O}({\bf G})$ when the rank of ${\bf G}$ is 1 or 2.
However when $n\geq 3$, the algebra $\mathscr{A}_n$ is of tame representation type. Moreover,
$\mathscr{A}_n$ is of wild type when $n\geq 4$.

\medskip
This paper is organized as follows:  Section 1 contains some preliminaries and  we also introduce the principal representation category $\mathscr{O}({\bf G})$ in this section.
In Section 2, we study the injective objects in
$\mathscr{O}({\bf G})$ and give some evidences to show that $\mathscr{O}({\bf G})$ may has enough injectives when $\Bbbk = \mathbb{C}$. Under the assumption that this conjecture is true, we show that $\mathscr{O}({\bf G})$ is
a highest weight category in Section 3. The algebra structure of $A_\theta$ is also  studied in this section. Section 4 is devoted to study  the algebras $\mathscr{A}_n$ which is isomorphic to $A_\theta$ when $|I(\theta)|=n$.

\bigskip
\noindent{\bf Acknowledgements}\ \  The author is grateful to Prof. Nanhua Xi for his constant encouragement and guidance.  The author would also like to thank  Prof. Ming Fang, Prof. Zongzhu Lin, Prof. Toshiaki Shoji and Dr. Xiaoyu Chen for their helpful discussions and comments.  Part of this work was done during the author's visit to Institute of Mathematics, Chinese Academy of Sciences. The author is grateful to the institute for hospitality.

%
%

\section{Principal Representation Category }
As in the introduction, let ${\bf G}$ be a connected reductive algebraic  group defined over $\mathbb{F}_q$ with the standard Frobenius map $F$ (e.g., $GL_n(\bar{\mathbb{F}}_q)$,  $SL_n(\bar{\mathbb{F}}_q)$, $SO_{2n}(\bar{\mathbb{F}}_q)$, $SO_{2n+1}(\bar{\mathbb{F}}_q)$, $Sp_{2n}(\bar{\mathbb{F}}_q)$,$\cdots$). Let ${\bf B}$ be an $F$-stable Borel subgroup, and ${\bf T}$ be an $F$-stable maximal torus contained in ${\bf B}$, and ${\bf U}=R_u({\bf B})$ be the ($F$-stable) unipotent radical of ${\bf B}$. We denote by $\Phi=\Phi({\bf G};{\bf T})$ the corresponding root system, and by $\Phi^+$ (resp. $\Phi^-$) the set of positive (resp. negative) roots determined by ${\bf B}$. Let $W=N_{\bf G}({\bf T})/{\bf T}$ be the corresponding Weyl group. One denotes by $\Delta=\{\alpha_i\mid i\in I\}$ the set of simple roots and $S=\{s_i\mid i\in I\}$ the corresponding simple reflections in $W$. For each $w\in W$, let $\dot{w}$ be one representative in $N_{\bf G}({\bf T})$. For any $w\in W$, let ${\bf U}_w$ (resp. ${\bf U}_w'$) be the subgroup of ${\bf U}$ generated by all ${\bf U}_\alpha$ (the root subgroup of $\alpha\in\Phi^+$) with $w\alpha\in\Phi^-$ (resp. $w\alpha\in\Phi^+$). The multiplication map ${\bf U}_w\times{\bf U}_w'\rightarrow{\bf U}$ is a bijection (see \cite[Proposition 2.5.12]{Ca}).

\medskip

 In this paper we fix an algebraically closed field $\Bbbk$ such that $\op{char}\Bbbk\neq \op{char} \bar{\mathbb{F}}_q$. Denote by
 $$\widehat{\bf T}=\{\theta \mid \theta : {\bf T}\rightarrow \Bbbk^* \ \text{is a group homomorphism}\}$$ the character group of ${\bf T}$ over ${\Bbbk}$. Each $\theta\in\widehat{\bf T}$ is regarded as a character of ${\bf B}$ by the homomorphism ${\bf B}\rightarrow{\bf T}$. Let ${\Bbbk}_\theta$ be the corresponding ${\bf B}$-module. The induced module $\mathbb{M}(\theta)=\Bbbk{\bf G}\otimes_{\Bbbk{\bf B}}{\Bbbk}_\theta$ was studied in the paper \cite{CD3} and all the composition factors of $\mathbb{M}(\theta)$ are given in \cite{CD3}. For convenience, we recall the main results here.

Let ${\bf 1}_{\theta}$ be a nonzero element in ${\Bbbk}_\theta$. We write $x{\bf 1}_{\theta}:=x\otimes{\bf 1}_{\theta}\in \mathbb{M}(\theta)$ for short.
It is clear  that $\displaystyle \mathbb{M}(\theta)=\sum_{w\in W}\Bbbk {\bf U}_{w^{-1}}\dot{w}{\bf 1}_{\theta}$ and the set $$\{ u \dot{w}{\bf 1}_{\theta} \mid  w\in W, u\in {\bf U}_{w^{-1}}\} $$ forms a basis of $\mathbb{M}(\theta)$ using Bruhat decomposition. As \cite[Proposition 2.2]{CD3} showed, the $\Bbbk {\bf G}$-module $\mathbb{M}(\theta)$ is indecomposable.

For each $i \in I$, let ${\bf G}_i$ be the subgroup of $\bf G$ generated by ${\bf U}_{\alpha_i}, {\bf U}_{-\alpha_i}$ and we set ${\bf T}_i= {\bf T}\cap {\bf G}_i$. For $\theta\in\widehat{\bf T}$, define the subset $I(\theta)$ of $I$ by $$I(\theta)=\{i\in I \mid \theta| _{{\bf T}_i} \ \text {is trivial}\}.$$
For $J\subset I(\theta)$, let ${\bf G}_J$ be the subgroup of $\bf G$ generated by ${\bf G}_i$, $i\in J$. We choose a representative $\dot{w}\in {\bf G}_J$ for each $w\in W_J$. Thus, the element $w{\bf 1}_\theta:=\dot{w}{\bf 1}_\theta$  $(w\in W_J)$ is well defined.

For $J\subset I(\theta)$, we set
$$\eta(\theta)_J=\sum_{w\in W_J}(-1)^{\ell(w)}w{\bf 1}_{\theta},$$
and let $\Delta(\theta)_J=\displaystyle \sum_{w\in W}\Bbbk {\bf U}\dot{w}\eta(\theta)_J$, which is a submodule of $\mathbb{M}(\theta)$. In particular, we have $\Delta(\theta)_J=\Bbbk{\bf G}\eta(\theta)_J$.
We define
$$E(\theta)_J=\Delta(\theta)_J/\Delta(\theta)_J',$$
where $\Delta(\theta)_J'$ is the sum of all $\Delta(\theta)_K$ with $J\subsetneq K\subset I(\theta)$.

For any subset $J\subset I$ and $K\subset I(\theta)$, we set
$$
\aligned
X_J &\ =\{x\in W\mid x~\op{has~minimal~length~in}~xW_J\};\\
Z_K &\ =\{w\in X_K \mid \mathscr{R}(ww_K)\subset K\cup (I\backslash I(\theta))\}.
\endaligned
$$
Then by \cite[Proposition 2.7]{CD3}, we have $$E(\theta)_J=\sum_{w\in Z_J}\Bbbk {\bf U}_{w_Jw^{-1}}\dot{w}C(\theta)_J.$$
where $C(\theta)_J$ is the image of $\eta(\theta)_J$ in $E(\theta)_J$. In particular, the following set  $$\{u\dot{w}C(\theta)_J \mid w\in Z_J, u\in {\bf U}_{w_Jw^{-1}} \}$$ forms a basis of $E(\theta)_J$.

The following theorem (see \cite[Theorem 3.1]{CD3}) gives all the composition factors of $\mathbb{M}(\theta)$ explicitly.

\begin{Thm}\label{EJ}
All the $\Bbbk {\bf G}$-modules $E(\theta)_J$ $(J\subset I(\theta))$ are irreducible and pairwise non-isomorphic. In particular, $\mathbb{M}(\theta)$ has exactly $2^{|I(\theta)|}$ composition factors with each of multiplicity one.
\end{Thm}

\medskip
The irreducible $\Bbbk {\bf G}$-modules $E(\theta)_J$ can also be realized by parabolic induction.
Let $\theta\in\widehat{\bf T}$ and $K\subset I(\theta)$. Since $\theta|_{{\bf T}_i}$ is trivial for all $i\in K$, it induces a character (still denoted by $\theta$) of $\overline{\bf T}={\bf T}/{\bf T}\cap[{\bf L}_K,{\bf L}_K]$. Therefore, $\theta$ is regarded as a character of ${\bf L}_K$ by the homomorphism ${\bf L}_K\rightarrow\overline{\bf T}$ (with the kernel $[{\bf L}_K,{\bf L}_K]$), and hence as a character of ${\bf P}_K$ by letting ${\bf U}_K$ acts trivially. We set $\mathbb{M}(\theta, K):=\Bbbk{\bf G}\otimes_{\Bbbk{\bf P}_K}\theta$.  Let ${\bf 1}_{\theta, K}$ be a nonzero element in the one-dimensional module $\Bbbk_\theta$ associated to $\theta$. We abbreviate $x{\bf 1}_{\theta, K}:=x\otimes{\bf 1}_{\theta, K} \in \mathbb{M}(\theta, K)$ as before.

Now set $\mathscr{R}(w)=\{i\in I\mid ws_i< w \}$ and $Y_K= \{ w\in W\mid \mathscr{R}(w)\subset I\backslash K \}$. Then by the same argument of \cite[Lemma 6.2]{D}, we have
$$\mathbb{M}(\theta, K)=\sum_{w\in Y_K}\Bbbk {\bf U}_{w^{-1}}\dot{w}{\bf 1}_{\theta, K} $$
and moreover, the following set $$\{ u \dot{w}{\bf 1}_{\theta, K}  \mid  w\in Y_K, u\in {\bf U}_{w^{-1}}\}$$ is a basis of
$\mathbb{M}(\theta, K)$.  One has that $\mathbb{M}(\theta, K)$ is
a indecomposable $\Bbbk {\bf G}$-module. Indeed, by the same discussion of  \cite[Proposition 2.2]{CD3}, we consider the endomorphism algebra of $\mathbb{M}(\theta, K)$ and have
$$\op{End}_{\bf G}(\mathbb{M}(\theta, K))\cong \op{Hom}_{{\bf P}_K}(\Bbbk_{\theta}, \mathbb{M}(\theta, K))\cong \Bbbk$$ which implies that  $\mathbb{M}(\theta, K)$ is indecomposable. Using the same proof of \cite[Theorem 6.3]{D} and \cite[Corollary 3.8]{CD1}, we have the following proposition.
\begin{Prop}\label{Parabolic}
Let $\theta\in\widehat{\bf T}$. For $K \subset I(\theta)$, we have $$\mathbb{M}(\theta, K) \cong \displaystyle \mathbb{M}(\theta)\Big{/} \sum_{s\in {K}}\Delta(\theta)_{\{s\}}.$$ Thus all the composition factors of $\mathbb{M}(\theta, K)$ are $E(\theta)_J$ with $J\subset I(\theta)\backslash K $.
\end{Prop}

For $J\subset I(\theta)$, set $J'=I(\theta)\backslash J$ and  we denote by $\nabla(\theta)_J= \mathbb{M}(\theta, J')=\Bbbk{\bf G}\otimes_{\Bbbk{\bf P}_{J'}}\Bbbk_\theta$.  Let $E(\theta)_J'$ be the submodule of $\nabla(\theta)_J$ generated by $$D(\theta)_J:=\sum_{w\in W_J}(-1)^{\ell(w)}\dot{w}{\bf 1}_{\theta, J'}.$$
We see that $E(\theta)_J'$ is isomorphic to $E(\theta)_J$ as $\Bbbk {\bf G}$-modules by \cite[Proposition 1.9]{CD3}. Therefore $E(\theta)_J$ can be regarded as the scole of  $\nabla(\theta)_J$.

\medskip

At the end of this section  we introduce a category $\mathscr{O}({\bf G})$ called principal representation category. It is the full subcategory of $\Bbbk{\bf G}$-Mod such that any object $M$ in $\mathscr{O}({\bf G})$ is of finite length and its composition factors are $E(\theta)_J$ for some $\theta \in \widehat{\bf T}$ and $J\subset I(\theta)$. Thus $\mathscr{O}({\bf G})$ is an abelian category.   The category $\mathscr{O}({\bf G})$ is obviously noetherian and artinian by its construction.
By the previous discussion we already have three interesting kinds of modules in $\mathscr{O}({\bf G})$, the irreducible modules $E(\theta)_J$, the  modules
$\Delta(\theta)_J$  and the modules  $\nabla(\theta)_J$. These modules are frequently used in the following sections.

\section{Injective Objects in $\mathscr{O}({\bf G})$ }

In this section, we assume that $\Bbbk$ is an algebraically closed field of characteristic $0$ (e.g., $\Bbbk = \mathbb{C}$). Under this assumption we will consider the injective objects in the principal representation category $\mathscr{O}({\bf G})$.
Firstly, for each $\theta\in\widehat{\bf T}$, the space $\op{Hom}_{\Bbbk {\bf B}}(\Bbbk {\bf G},  \Bbbk_{\theta})$ is a $\Bbbk {\bf G}$-module whose module structure  is given by $$(g \varphi) (x)= \varphi (xg), \ \  \text{where} \ \varphi \in \op{Hom}_{\Bbbk {\bf B}}(\Bbbk {\bf G},  \Bbbk_{\theta}),  g\in {\bf G}, x\in \Bbbk {\bf G}.$$
Since this space $\op{Hom}_{\Bbbk {\bf B}}(\Bbbk {\bf G},  \Bbbk_{\theta})$ has uncountable dimension, it is not an object in $\mathscr{O}({\bf G})$ generally. However we have
\begin{Lem} \label{injective} For  each $\theta\in\widehat{\bf T}$, $\op{Hom}_{\Bbbk {\bf B}}(\Bbbk {\bf G},  \Bbbk_{\theta})$ is an injective $\Bbbk {\bf G}$-module.
\end{Lem}
\begin{proof} Using the setting and properties in \cite[Section 2]{FS}, when $\Bbbk$ is an algebraically closed field of characteristic zero, $\Bbbk {\bf B}$ is a locally Wedderburn algebra. Thus by \cite[Lemma 3]{FS}, we see that $\Bbbk_{\theta}$ is an injective $\Bbbk {\bf B}$-module. Therefore  $\op{Hom}_{\Bbbk {\bf B}}(\Bbbk {\bf G},  \Bbbk_{\theta})$ is an injective $\Bbbk {\bf G}$-module for each $\theta\in\widehat{\bf T}$.
\end{proof}

\begin{Rem} According to  in \cite[Theorem 1]{FS},  the trivial $\Bbbk {\bf T}$-module is injective if and only if the order of no elements in ${\bf T}$ vanishes in $\Bbbk$. Thus if $\op{char} \Bbbk > 0$, the trivial $\Bbbk {\bf T}$-module is not injective in general. Then the trivial $\Bbbk {\bf B}$-module is also not injective. However I guess that the $\Bbbk {\bf G}$-module $\op{Hom}_{\Bbbk {\bf B}}(\Bbbk {\bf G},  \Bbbk_{\theta})$  can also be injective. We need to prove it in other methods.

\end{Rem}

\medskip
 Using the Bruhat decomposition, each element $g\in {\bf G}$ has the unique expression  $g=b\dot{w}u$ for some  $w\in W$, $b\in {\bf B}$ and $u\in {\bf U}$.
For $\theta\in\widehat{\bf T}$  and an element $w\in W$,  we consider the function
$\rho_{ w, \theta}\in \op{Hom}_{\Bbbk {\bf B}}(\Bbbk {\bf G},  \Bbbk_{\theta})$ defined by
$$\rho_{ w, \theta}(b\dot{w}'u)=\delta_{w,w'}\theta(b)$$
where $\delta_{w,w'}$ is the kronecker symbol. Now we let $$\mathbb{X}_{\theta}=\sum_{w\in W}\Bbbk {\bf G} \rho_{ w, \theta}$$
which is the $\Bbbk {\bf G}$-submodule generated by $\rho_{ w, \theta}, w\in W$.

The Weyl group $W$ acts naturally on $\widehat{\bf T}$ by
$$(w\cdot \theta ) (t):=\theta^w(t)=\theta(\dot{w}t\dot{w}^{-1})$$
for any $\theta\in\widehat{\bf T}$. Denote by
$W_{\theta}$ the stabilizer of $\theta\in\widehat{\bf T}$. Thus the parabolic subgroup $W_{I(\theta)}$ is a subgroup of $W_{\theta}$. For any $w_1, w_2 \in W$, we see that $\theta^{w_1}= \theta^{w_2} $ if and only if $w_1w_2^{-1} \in W_{\theta}$.

\begin{Lem} \label{generatedmodule}
One has that $\Bbbk {\bf G} \rho_{ w, \theta}\cong \mathbb{M}(\theta^{w})$ as $\Bbbk {\bf G}$-modules for each $w\in W$.

\end{Lem}

\begin{proof} Firstly using Frobenius reciprocity we have
$$ \op{Hom}_{\Bbbk {\bf G}} (\mathbb{M}(\theta^w) , \Bbbk {\bf G} \rho_{ w, \theta}) \cong \op{Hom}_{\Bbbk {\bf B}}({\bf 1}_{\theta^w}, \Bbbk {\bf G} \rho_{ w, \theta} ) \cong \Bbbk.$$
Thus it is not difficult to see that $\Bbbk {\bf G} \rho_{ w, \theta}$ is a quotient $\Bbbk {\bf G}$-module of $\mathbb{M}(\theta^w)$.
Since the $\Bbbk {\bf G}$-module $\Delta(\theta)_{I(\theta^w)}$ is the socle of $\mathbb{M}(\theta^w)$, to get $\Bbbk {\bf G} \rho_{ w, \theta}\cong \mathbb{M}(\theta^w)$ as $\Bbbk {\bf G}$-modules, it is enough to show that the element  $$\sum_{x\in W_{I(\theta^w)}}(-1)^{\ell(x)}\dot{x}\rho_{ w, \theta}\ne 0.$$
Indeed, we have
$$\sum_{x\in W_{I(\theta^w)}}(-1)^{\ell(x)}\dot{x}\rho_{ w, \theta} (b\dot{w})=\sum_{x\in W_{I(\theta^w)}}(-1)^{\ell(x)}\rho_{ w, \theta} (b\dot{w} \dot{x}) \ne 0.$$
Thus the lemma is proved. In particular,  we have $\Bbbk {\bf G} \rho_{ e, \theta}\cong \mathbb{M}(\theta)$ as $\Bbbk {\bf G}$-modules, where $e$ is the neutral element of $W$.

\end{proof}

For a fixed $\theta\in\widehat{\bf T}$ and  $w\in W$, let $J_{w,\theta}= \mathscr{R}(w) \cap  I(\theta^w)$. Then $w$ has the unique form $w=x\vartheta_{J_{w,\theta}}$, where $\vartheta_{J_{w,\theta}}$ is the longest element in $W_{J_{w,\theta}}$.  In this case we set
$$\varsigma_{w,\theta}= \sum_{v\in W_{J_{w,\theta}}}\rho_{ xv, \theta} \ \ \text{and} \ \  \mathbb{X}_{w,\theta}= \Bbbk {\bf G} \varsigma_{w,\theta} .$$

\begin{Lem} \label{direct summand}
For each $\theta\in\widehat{\bf T}$ and  $w\in W$, one has that
$$ \mathbb{X}_{w,\theta} \cong \Bbbk{\bf G}\otimes_{\Bbbk{\bf P}_{J_{w,\theta}}}\Bbbk_{\theta^w}$$ as $\Bbbk {\bf G}$-modules. Then we have $ \mathbb{X}_{\theta} \displaystyle \cong \bigoplus_{w\in W}\mathbb{X}_{w,\theta} $ as $\Bbbk {\bf G}$-modules.

\end{Lem}

\begin{proof} By the setting of $\varsigma_{w,\theta}$, it  is easy to check that
$t \varsigma_{w,\theta} =\theta^w(t) \varsigma_{w,\theta}$ for any $t\in {\bf T}.$ Moreover for any $z\in W$ and  $s\in J_{z,\theta}$, we have
$$\dot{s}(\rho_{ z, \theta} + \rho_{ zs, \theta} ) =\rho_{ z, \theta}  + \rho_{ zs, \theta}.$$
Thus we get $\dot{s}\varsigma_{w,\theta}=\varsigma_{w,\theta}$ for any $s \in J_{w,\theta}$. Hence
 $ \mathbb{X}_{w,\theta}$ is a quotient $\Bbbk {\bf G}$-module of $\Bbbk{\bf G}\otimes_{\Bbbk{\bf P}_{J_{w,\theta}}}\Bbbk_{\theta^w}$. Noting that $E(\theta^w)_{I(\theta^w)\setminus{J_{w,\theta}}}$ is the socle of $\Bbbk{\bf G}\otimes_{\Bbbk{\bf P}_{J_{w,\theta}}}\Bbbk_{\theta^w}$, by the same discussion as Lemma \ref{generatedmodule}, we get $$ \mathbb{X}_{w,\theta} \cong \Bbbk{\bf G}\otimes_{\Bbbk{\bf P}_{J_{w,\theta}}}\Bbbk_{\theta^w}$$ as $\Bbbk {\bf G}$-modules. Therefore each $\mathbb{X}_{w,\theta} $ is a indecomposable $\Bbbk {\bf G}$-module.

 It is easy to verify that
$ \mathbb{X}_{\theta} = \displaystyle \sum_{w\in W}\mathbb{X}_{w,\theta} $ as $\Bbbk {\bf G}$-modules. Indeed, it is obvious to see that $\displaystyle \sum_{w\in W}\mathbb{X}_{w,\theta} \subseteq \mathbb{X}_{\theta} $.  On the other hand, we can do induction on the length of the elements in $W$ to show that $\displaystyle \rho_{z, \theta} \in \sum_{w\in W}\mathbb{X}_{w,\theta}$ for any $z\in W$. Noting that $\mathbb{X}_{\theta}$ is generated by $\rho_{ w, \theta}$ with $w\in W$. Then we have $ \mathbb{X}_{\theta} = \displaystyle \sum_{w\in W}\mathbb{X}_{w,\theta} $.

Given an order of the elements in $$W=\{w_1, w_2, \dots, w_n\}$$
such that $\ell(w_i)\leq \ell(w_j)$ for $i<j$. In the following we show that
$$\mathbb{X}_{w_k ,\theta}  \cap  (\sum_{i< k}\mathbb{X}_{w_i ,\theta})= 0$$
for any $k=1,2,\dots, n$, which implies that $\displaystyle  \mathbb{X}_{\theta} \cong \bigoplus_{w\in W}\mathbb{X}_{w,\theta} $ as $\Bbbk {\bf G}$-modules. Now suppose that there exists an integer $k$ such that
$$\mathbb{X}_{w_k ,\theta}  \cap  (\sum_{i< k}\mathbb{X}_{w_i ,\theta})\ne  0.$$
For convenience we set $\Gamma= I(\theta^{w_k})\setminus{J_{w_k,\theta}}$. Using $\mathbb{X}_{w_k ,\theta}  \cong \Bbbk{\bf G}\otimes_{\Bbbk{\bf P}_{J_{w_k,\theta}}}\Bbbk_{\theta^{w_k}} $,
one has that $E(\theta^{w_k})_{\Gamma}$ is the socle of  $\mathbb{X}_{w_k ,\theta}$.
 Thus we get
$$  \sum_{x\in W_\Gamma}(-1)^{\ell(x)}\dot{x} \varsigma_{w_k,\theta}  \in \sum_{i< k}\mathbb{X}_{w_i ,\theta}.$$
However we have the following  fact that for any $w\ne w'$ with $\ell(w)\leq \ell(w')$, there exists an element $u\in {\bf U}$ such that
$\dot{x} \rho_{w,\theta}(\dot{w}' u)\ne 0$ for any $x\in W$. Using this fact we get  a contradiction. Actually, for any element $\varphi \in \displaystyle \sum_{i< k}\mathbb{X}_{w_i ,\theta}$, there exists an element $u\in {\bf U}$ such that
$$\sum_{x\in W_\Gamma}(-1)^{\ell(x)}\dot{x} \varsigma_{w_k,\theta}(\dot{w_k} u)\ne 0 \ \ \text{and} \ \ \  \varphi(\dot{w_k} u)=0. $$
Thus the lemma is proved.

\end{proof}

\begin{Prop}
Let $\theta\in\widehat{\bf T}$, we have
$$  \op{Hom}_{\Bbbk {\bf G}}(M,  \mathbb{X}_{\theta} ) \cong \op{Hom}_{\Bbbk {\bf G}} (M, \op{Hom}_{\Bbbk {\bf B}}(\Bbbk {\bf G},  \Bbbk_{\theta}))$$
for any simple object $M\in \mathscr{O}({\bf G})$.
\end{Prop}

 \begin{proof} Each simple object in $\mathscr{O}({\bf G})$ has the form $E(\lambda)_J$ for some $\lambda \in \widehat{\bf T}$ and $J\subset I(\lambda)$. Firstly, we have
$$\op{Hom}_{\Bbbk {\bf G}} (E(\lambda)_J, \op{Hom}_{\Bbbk {\bf B}}(\Bbbk {\bf G},  \Bbbk_{\theta})) \cong \op{Hom}_{\Bbbk {\bf B}} (E(\lambda)_J,  \Bbbk_{\theta})$$
by Frobenius reciprocity.
Since $E(\lambda)_J= \displaystyle \sum_{w\in Z_J} \Bbbk {\bf U}_{w_Jw^{-1}}w C(\lambda)_J$, we have
$$E(\lambda)_J\cong \bigoplus_{w\in Z_J} \Bbbk {\bf U}_{w_Jw^{-1}}w C(\lambda)_J$$
as $\Bbbk {\bf B}$-modules. Then we get
$$\dim \op{Hom}_{\Bbbk {\bf B}} (E(\lambda)_J,  \Bbbk_{\theta}) =\sharp \{w \in Z_J\mid \lambda=\theta^w \}.$$

On the other hand, by Lemma \ref{direct summand}, we have $$ \op{Hom}_{\Bbbk {\bf G}}(E(\lambda)_J,  \mathbb{X}_{\theta} )\cong  \bigoplus_{w\in W}\op{Hom}_{\Bbbk {\bf G}}(E(\lambda)_J,  \mathbb{X}_{w, \theta} ), $$
which implies that
$$\dim \op{Hom}_{\Bbbk {\bf G}}(E(\lambda)_J,  \mathbb{X}_{\theta} )=\sharp \{w \in W \mid \lambda=\theta^w \ \text{and}  \ J_{w,\theta}= I(\lambda)\setminus J  \}$$

For  fixed $\theta \in \widehat{\bf T}$, $\lambda \in \widehat{\bf T}$ and $J\subset I(\lambda)$,  we let $$\Omega_{\theta}(\lambda, J)= \{w\in W \mid \lambda=\theta^w \ \text{and}\ J\subset \mathscr{R}(w) \subset J\cup (I\setminus I(\lambda))  \}.$$ Then it is easy to check that
$$\dim \op{Hom}_{\Bbbk {\bf B}} (E(\lambda)_J, \Bbbk_{\theta}) = |\Omega_{\theta}(\lambda, J)|,$$
 $$ \dim \op{Hom}_{\Bbbk {\bf G}}(E(\lambda)_J,  \mathbb{X}_{\theta} )= |\Omega_{\theta}(\lambda,  I(\lambda) \setminus J)|. $$
In the following we will show that $|\Omega_{\theta}(\lambda, J)|= |\Omega_{\theta}(\lambda,  I(\lambda) \setminus J)|$.

Let $x_1, x_2, \dots, x_m$ be a complete representative set of the left cosets of $W_{I(\lambda)}$ in $W$. For each $w\in \Omega_{\theta}(\lambda, J)$ with the form $w=x_i y$ for some $i=1,2,\dots, m$ and $y\in W_{I(\lambda)}$.  Since $\theta^w =\theta^{x_i y}=\lambda$ and $y\in W_{I(\lambda)}$,  we get $\theta^{x_i}=\lambda$.
Now we set $\Xi(w)=x_i w_{I(\lambda)}y$. Then we also have $\theta^{\Xi(w)}=\theta^{x_i}=\lambda$. Thus it is easy to check that
$$\Xi(w)=x_i w_{I(\lambda)}y \in \Omega_{\theta}(\lambda,  I(\lambda) \setminus J)$$
which gives a bijecction $\Xi: \Omega_{\theta}(\lambda, J)\longrightarrow \Omega_{\theta}(\lambda,  I(\lambda) \setminus J) $. Hence we have
$$  \op{Hom}_{\Bbbk {\bf G}}(E(\lambda)_J,  \mathbb{X}_{\theta} ) \cong \op{Hom}_{\Bbbk {\bf G}} (E(\lambda)_J, \op{Hom}_{\Bbbk {\bf B}}(\Bbbk {\bf G},  \Bbbk_{\theta}))$$
by the above discussion.

\end{proof}

\begin{Conj} \label{conjecture}
Let $\theta\in\widehat{\bf T}$, one has that
$$  \op{Hom}_{\Bbbk {\bf G}}(M,  \mathbb{X}_{\theta} ) \cong \op{Hom}_{\Bbbk {\bf G}} (M, \op{Hom}_{\Bbbk {\bf B}}(\Bbbk {\bf G},  \Bbbk_{\theta}))$$
for any object $M\in \mathscr{O}({\bf G})$. Thus $\mathbb{X}_{\theta} $ is an injective object in $\mathscr{O}({\bf G})$ for each $\theta\in \widehat{\bf T}$.
\end{Conj}

By Lemma \ref{direct summand}, each $ \mathbb{X}_{w,\theta}$ is a direct summand of $ \mathbb{X}_{\theta}$, thus  $\nabla(\theta)_J=\Bbbk{\bf G}\otimes_{\Bbbk{\bf P}_{J'}}\Bbbk_\theta$ is an injective object in $\mathscr{O}({\bf G})$ for any $J\subset I(\theta)$.  So $\nabla(\theta)_J$ is the injective envelope of $E(\theta)_J$.
Therefore Conjecture \ref{conjecture} suggests the existence of enough injectives in the category $\mathscr{O}({\bf G})$.

\begin{Thm} \label{enough injectives}
The category $\mathscr{O}({\bf G})$ has enough injectives.
\end{Thm}
\begin{proof} We do induction on the length of $M\in \mathscr{O}({\bf G})$. Firstly, $\nabla(\theta)_J= \Bbbk{\bf G}\otimes_{\Bbbk{\bf P}_{J'}}\Bbbk_\theta$ is the injective envelope of
$E(\theta)_J$. Assuming that $M$ has length bigger than $1$, it has a simple quotient $E(\theta)_J$.
 Then we get a short exact sequence
$$0\rightarrow N \xrightarrow{f} M \xrightarrow{g} E(\theta)_J \rightarrow 0.$$ So by induction there exists a monomorphism $N\xrightarrow{i} Q$ for some injective $Q\in \mathscr{O}({\bf G})$. Then there exists a morphism $M\xrightarrow{h} Q$ such that $i=h f$. Therefore either $M\xrightarrow{h} Q$ is a monomorphism or
$M\cong N\oplus E(\theta)_J$. In the latter case, $M\rightarrow Q\oplus \nabla(\theta)_J$ is momomorphism and the theorem is proved.

\end{proof}

\section{Highest Weight Category $\mathscr{O}({\bf G})$ }

From now on, we assume that Conjecture \ref{conjecture}  is established.  Thus $\mathscr{O}({\bf G})$ has enough injectives. We will show that the principal representation category $\mathscr{O}({\bf G})$ is a highest weight category.
Firstly we recall the definition of highest weight categories (see \cite{CPS}).

\begin{Def}\label{HWC} Let $\mathscr{C}$ be a locally artinian, abelian, $\Bbbk$-linear category with enough injectives that satisfies Grothendieck's condition. Then we call $\mathscr{C}$ a highest weight category if there exists a locally finite poset $\Lambda$ (the "weights" of $\mathscr{C}$), such that:

(a) There is a complete collection $\{S(\lambda)_{\lambda \in \Lambda}\}$ of non-isomorphic simple objects of $\mathscr{C}$ indexed by the set $\Lambda$.

(b) There is a collection $\{A(\lambda)_{\lambda \in \Lambda}\}$ of objects of $\mathscr{C}$ and, for each $\lambda$, an embedding $S(\lambda)\subset A(\lambda) $ such that all composition factors $S(\mu)$ of $A(\lambda)/S(\lambda)$ satisfy $\mu < \lambda$. For $\lambda,\mu \in \Lambda$,
we have that $dim_{\Bbbk}\op{Hom}_{\mathscr{C}}(A(\lambda), A(\mu))$ and $[A(\lambda): S(\mu)]$ are finite.

(c) Each simple object $S(\lambda)$ has an injective envelope $I(\lambda)$ in $\mathscr{C}$.
Also, $I(\lambda)$  has a good filtration $0= F_0(\lambda)\subset F_1(\lambda)\subset \dots $ such that:

\noindent (i) $F_1(\lambda)\cong A(\lambda)$;

\noindent  (ii) for $n>1$, $F_n(\lambda)/F_{n-1}(\lambda) \cong A(\mu)$ for some $\mu=\mu(n)> \lambda$;

\noindent  (iii) for a given $\mu \in \Lambda$, $\mu=\mu(n)$ for only finitely many $n$;

\noindent  (iv) $\bigcup F_i(\lambda)= I(\lambda)$.
\end{Def}

Now we show that $\mathscr{O}({\bf G})$ is a highest weight category. In Definition \ref{HWC},  the set of  weights  is $\Lambda= \{(\theta, J )\ | \ \theta \in \widehat{\bf T}, J \subset I(\theta) \}$ and we define the order by
$$(\theta_1, J_1) \leq (\theta_2, J_2 ), \ \text{if}\ \theta_1=\theta_2 \ \text{and} \ J_1\supseteq J_2.$$
Set $S(\lambda)=A(\lambda)=E(\theta)_J$ and $I(\lambda)= \nabla(\theta)_J$, then the condition in Definition \ref{HWC} is easy to check. So  $\mathscr{O}({\bf G})$ is a highest weight category.

The highest weight category has many good properties (see \cite{CPS}). We list some interesting propositions of the category $\mathscr{O}({\bf G})$.
\begin{Prop}\label{Extension} \label{Ext} (1) For $n\geq 0$, $\op{Ext}^n_{\mathscr{O}({\bf G})}(M,N)$ is finite-dimensional for all $M,N\in \mathscr{O}({\bf G})$.

\noindent $(2)$ If $\op{Ext}^n_{\mathscr{O}({\bf G})}(E(\lambda)_J,E(\mu)_K)\ne 0$, then $\lambda=\mu$ and $J\subseteq K$. Moreover,  if $n>0$, we have   $\lambda=\mu$ and $J\subsetneq K$.
\end{Prop}

According to \cite[Corollary 5.6]{CD3}, we know that any finite-dimensional irreducible representations of $\bf G$ are one-dimensional when $\Bbbk$ is algebraically closed with $\op{char}\Bbbk\neq\op{char}\bar{\mathbb{F}}_q$. Moreover, these irreducible representations are isomorphic to $E(\theta)_\varnothing$ for some $\theta\in\widehat{\bf T}$ with $I(\theta)=I$. By Proposition \ref{Extension} we have the following corollary immediately.

\begin{Cor}\label{fdrep} The finite-dimensional complex irreducible representations of the group ${\bf G}$ is one-dimensional and all the finite-dimensional complex  representations of  ${\bf G}$ are semisimple.
\end{Cor}

For $\theta\in \widehat{\bf T}$, let $\mathscr{O}({\bf G})_\theta$ be the subcategory of $\mathscr{O}({\bf G})$ containing the objects whose subquotients are $E(\theta)_J$ for some $J\subset I(\theta)$.
Then by Proposition \ref{Ext}, we have $\mathscr{O}({\bf G})= \displaystyle \bigoplus_{\theta\in \widehat{\bf T}} \mathscr{O}({\bf G})_\theta$. For each $\theta\in \widehat{\bf T}$, $\mathscr{O}({\bf G})_\theta$ is a highest weight category and then there exists a finite-dimensional quasi-hereditary algebra $A_\theta$ such that $\mathscr{O}({\bf G})_\theta$ is equivalent to the right $A_\theta$-modules.
Indeed, if we set $ \mathscr{I}_{\theta}=\displaystyle\bigoplus_{J\subset I(\theta)}\nabla(\theta)_J$, then $A_\theta\cong \op{End}_{\mathscr{O}({\bf G})}(\mathscr{I}_{\theta})$. The functor $\op{Hom}_{\bf G}(-, \mathscr{I}_{\theta})^*$ form $\mathscr{O}({\bf G})_\theta$ to the right $A_\theta$-modules is an equivalence of categories. Therefore we also see that  the category $\mathscr{O}({\bf G})$ is a Krull-Schmidt category.

\medskip

Now we want to understand the structure of the algebra $A_\theta$, we have
$$A_\theta\cong \op{End}_{\mathscr{O}({\bf G})}(\mathscr{I}_{\theta})\cong \bigoplus_{J\subset I(\theta)}\op{Hom}_{\mathscr{O}({\bf G})}(\nabla(\theta)_ J, \mathscr{I}_{\theta} ).$$
The composition factors of  $\nabla(\theta)_ J= \mathbb{M}(\theta, J')$ with $J'=I(\theta)\backslash J$ are given in Proposition \ref{Parabolic}, therefore
$$
\aligned  \op{Hom}_{\mathscr{O}({\bf G})}(\nabla(\theta)_ J, \mathscr{I}_{\theta} )
 = &\  \op{Hom}_{\mathscr{O}({\bf G})}(\nabla(\theta)_ J, \bigoplus_{K\subset I(\theta) }\nabla(\theta)_ K)\\
= &\  \bigoplus_{K\subset J}\op{Hom}_{\mathscr{O}({\bf G})}(\nabla(\theta)_ J, \nabla(\theta)_ K)\\
\cong &\ \bigoplus_{K\subset J} \op{Hom}_{{\bf P}_{J'}}({\bf 1}_{\theta, J'} ,\mathbb{M}(\theta, K')).
\endaligned
$$
 So if let $\varphi_{K\subset J}$ be a $\Bbbk {\bf G}$-module morphism such that $\varphi_{K\subset J} ({\bf 1}_{\theta, J'})= {\bf 1}_{\theta, K'}$, then we have
 $$\op{Hom}_{\mathscr{O}({\bf G})}(\nabla(\theta)_ J, \nabla(\theta)_ K) \cong \Bbbk \varphi_{K\subset J}.$$ Now the $\Bbbk$-algebra $A_\theta$ has a
 $\Bbbk$ -basis $\{\varphi_{K\subset J}\mid K\subset J \subset I(\theta)\}$
and the multiplications are given by
$$\varphi_{K\subset M} \varphi_{L\subset J} =\left\{
\begin{array}{ll}
\varphi_{K\subset J} &\ \mbox{if}~L=M,\\
0 &\ \mbox{otherwise.}
\end{array}\right.$$

\begin{Cor} If $|I(\lambda)|=|I(\mu)|$, then $A_\lambda \cong A_\mu$ as $\Bbbk$-algebras and
thus $\mathscr{O}({\bf G})_\lambda$ is equivalent to  $\mathscr{O}({\bf G})_\mu$.
\end{Cor}

For each $\theta\in \widehat{\bf T}$, the right $A_\theta$-modules has enough projectives. Then $\mathscr{O}({\bf G})_\theta$ and hence $\mathscr{O}({\bf G})$ also have enough projectives. Moreover we have the following proposition.

\begin{Prop}\label{enough projectives} For $\theta\in \widehat{\bf T}$ and $J\subset I(\theta)$, $\Delta(\theta)_J$ is the projective cover of $E(\theta)_J$.
\end{Prop}

\begin{proof} The functor $\op{Hom}_{\bf G}(-, \mathscr{I}_{\theta})^*$ form $\mathscr{O}({\bf G})_\theta$
to the right $A_\theta$-modules  keeps the projectives.
So it is enough to show that $\op{Hom}_{\bf G}(\Delta(\theta)_J , \mathscr{I}_{\theta})^*$ is a projective right $A_\theta$-module.

We  denote $\varphi_J:=\varphi_{K\subset J}$ when $J=K$. As a right $A_\theta$-module, $A_\theta$
has a decomposition $A_\theta=\displaystyle \bigoplus_{J\subset I(\theta)} \varphi_J A_\theta$ and each $\varphi_J A_\theta$ is indecomposable projective. In the following we will show that $\op{Hom}_{\bf G}(\Delta(\theta)_J , \mathscr{I}_{\theta})^* \cong \varphi_J A_\theta$.

All the composition factors of $\Delta(\theta)_J$ are $E(\theta)_K$ with $J\subset K \subset I(\theta)$. Thus
$$
\aligned \op{Hom}_{\bf G}(\Delta(\theta)_J , \mathscr{I}_{\theta})&\
 =  \bigoplus_{K\subset I(\theta)}\op{Hom}_{\bf G}(\Delta(\theta)_J , \nabla(\theta)_K)\\
&\ = \bigoplus_{J\subset K\subset I(\theta)}\op{Hom}_{\bf G}(\Delta(\theta)_J , \nabla(\theta)_K).
\endaligned
$$
Let $f_{K \supset J}$ be the $\Bbbk {\bf G}$-module morphism such that
$f_{K \supset J}(\eta(\theta)_J)= {\bf 1}_{\theta, K'},$ where $K'=I(\theta)\backslash K$. Then we get
$$\op{Hom}_{\bf G}(\Delta(\theta)_J , \nabla(\theta)_K)\cong \Bbbk f_{K \supset J}.$$ Thus $\{f_{K \supset J} \mid J\subset K\subset I(\theta)\}$ is a basis of $\op{Hom}_{\bf G}(\Delta(\theta)_J , \mathscr{I}_{\theta})$.
The operation of $A_\theta$ on this basis is:
$$\varphi_{L\subset M} f_{K\supset J} =\left\{
\begin{array}{ll}
f_{L\supset J} &\ \mbox{if}~K=M \ \text{and}\  L\supset J, \\
0 &\ \mbox{otherwise.}
\end{array}\right.$$

Let $\widetilde{f}_{K\supset J}$ be the dual basis of $f_{K\supset J}$ and set $\widetilde{f}_{J}=\widetilde{f}_{K\supset J}$ when $J=K$. It is easy to check that the homomorphism
send $\widetilde{f}_{J}$ to $\varphi_J$ induces an isomorphism
$$\op{Hom}_{\bf G}(\Delta(\theta)_J , \mathscr{I}_{\theta})^* \cong \varphi_J A_\theta$$ as right $A_\theta$-modules.
Therefore  $\Delta(\theta)_J$ is the projective cover of $E(\theta)_J$ and the proposition is proved.
\end{proof}

\section{The Algebras $\mathscr{A}_n$}
Assume that $X$ is a finite set with  cardinality $|X|=n\geq 1$.
Denote by $M_{2^n}(\Bbbk)$ the matrix algebra over $\Bbbk$. The rows and columns of a matrix in this algebra are indexed by the subsets of $X$.
Fix an order of the subsets of $X$, we let $$\mathscr{A}_n= \{ (a_{Y,Z})\in M_{2^n}(\Bbbk) \mid a_{Y,Z}=0 \ \text{if}~Y ~\text{is not a subset of } Z \}$$
which is a subalgebra of the matrix algebra $ M_{2^n}(\Bbbk)$ whose rows and columns are indexed by the subsets of $X$.  The algebra $\mathscr{A}_n$ just depends on the cardinality of $X$.
Let $\{e_{Y,Z} \mid Y\subset Z\subset X\}$ be the standard basis of  $\mathscr{A}_n$. Without lost of generality, we can assume each element in $\mathscr{A}_n$ has the form of a upper triangular matrix. The Jacobson radical of $\mathscr{A}_n$ is
$\text{Rad}(\mathscr{A}_n)=\displaystyle \sum_{Y\varsubsetneq Z}\Bbbk e_{Y,Z}$. As a right $\mathscr{A}_n$-module, $\mathscr{A}_n$ has a decomposition
$$\mathscr{A}_n= \bigoplus_{Y\subset X}e_{Y,Y} \mathscr{A}_n$$
and for each $Y\subset X$, $e_{Y,Y} \mathscr{A}_n$ is a indecomposable projective module.
For each integer $n$, it is not difficult to see that $\mathscr{A}_n$
is a basic and connected algebra.
Therefore there exists a bound quiver $(\mathcal{Q}, \mathcal{I})$ such that $\Bbbk \mathcal{Q} /\mathcal{I} \cong \mathscr{A}_n $ (see \cite[Chapter II]{ASS}).

As a matter of fact, we can construct the bound quiver $(\mathcal{Q}, \mathcal{I})$  associate to $X$ with $|X|=n$. The vertices $Q_0$ of the quiver  $\mathcal{Q}=(Q_0, Q_1, s,t)$  is indexed by the subsets of $X$ and we  denote it  by $Q_0=\{i_Y\mid Y\subset X\}$. For two vertices $i_Y, i_Z \in Q_0$, there exists an edge $\alpha \in Q_1$ between them if  $Y\subset Z$ with $|Z\backslash Y|=1$ and the orientation is given by $s(\alpha)=i_Y$, $t(\alpha)=i_Z$. We denote such arrow by $\alpha_{Y,Z}$.
The admissible ideal $\mathcal{I}$ in $\Bbbk \mathcal{Q}$ is generated by all elements $\omega_1-\omega_2$
given by the pairs $\{\omega_1,\omega_2\}$ of paths in $\mathcal{Q}$ having the same starting and ending vertices.
Then we have $\Bbbk \mathcal{Q} /\mathcal{I} \cong \mathscr{A}_n $. It is not difficult to see that the algebra $A_\theta$ is isomorphic to the  algebra $\mathscr{A}_n$ with $n=|I(\theta)|$.

\medskip

For $n=1$, the algebra $\mathscr{A}_1$ is the path algebra of the Dynkin quiver of type $A_2$. The number of isomorphism classes of indecomposable representations of $\mathscr{A}_1$ is $3$. By the correspondence of
$\mathscr{O}({\bf G})_\theta$ and the right $A_\theta$-modules under the assumption $\Bbbk= \mathbb{C}$, then for ${\bf G}=SL_2$, the indecomposable modules in $\mathscr{O}({\bf G})$ is $\mathbb{M}(\op{tr}), \op{St}, \Bbbk_{\op{tr}}$ and $\{\mathbb{M}(\theta)\mid \theta \in \widehat{\bf T} \ \text{nontrivial}\}$. Therefore each module in $\mathscr{O}({\bf G})$ is a direct sum of these modules.

\medskip

For $n=2$, the incidence algebra $\mathscr{A}_2$ is isomorphic the algebra given by  the following quiver
\begin{center}
\begin{tikzpicture}[scale=1]
\draw  (4,0) node (I11) {$a$} +(1.2,0.6) node (I12) {$b$} +(1.2,-0.6) node (I21) {$c$} + (2.4,0) node (I22) {$d$};
\draw[->] (I11)--(I12) node[pos=.5,above] {$\alpha$};
\draw[->] (I12)--(I22) node[pos=.5,above] {$\beta$};
\draw[->] (I11)--(I21) node[pos=.5,above] {$\gamma$};
\draw[->] (I21)--(I22) node[pos=.5,above] {$\delta$};
\end{tikzpicture}
\end{center}
bound by the relation $\beta \alpha= \delta \gamma$. Therefore the  representations of the algebra $\mathscr{A}_2$ is given  by the following diagram

\begin{center}
\begin{tikzpicture}[scale=1]
\draw  (4,0) node (I11) {$V_a$} +(2,1) node (I12) {$V_b$} +(2,-1) node (I21) {$V_c$} + (4,0) node (I22) {$V_d$};
\draw[->] (I11)--(I12) node[pos=.5,above] {$f_{ab}$};
\draw[->] (I12)--(I22) node[pos=.5,above] {$f_{bd}$};
\draw[->] (I11)--(I21) node[pos=.5,above] {$f_{ac}$};
\draw[->] (I21)--(I22) node[pos=.5,above] {$f_{cd}$};
\end{tikzpicture}
\end{center}
such that $V_a$, $V_b$, $V_c$, $V_d$ are vector spaces and $f_{ab}$, $f_{bd}$, $f_{ac}$, $f_{cd}$ are linear morphisms which satisfy $f_{bd} f_{ab}= f_{cd}f_{ac}$. The Auslander-Reiten quiver of the algebra $\mathscr{A}_2$ was known in the example of Chap VII.2 of  the book \cite{ARS}. Thus the algebra $\mathscr{A}_2$  is of finite type and there are 11 indecomposable $\mathscr{A}_2$-modules up to isomorphism. For an given representation $\mathbb{V}$ of $\mathscr{A}_2$, the dimension vector of  $\mathbb{V}$ is denoted by  $$\underline{\text{dim}}\mathbb{V}= (\text{dim}V_a, \text{dim}V_b, \text{dim}V_c, \text{dim}V_d ).$$
Hence all the indecomposable $\mathscr{A}_2$-modules (up to isomorphism) are replaced by their dimension vectors which are the following:
$$
\aligned
&\ (1,0,0,0),\ (0,1,0,0),\ (0,0,1,0),\ (0,0,0,1),\ (1,1,0,0),\ (1,0,1,0), \\
&\ (0,1,0,1),\ (0,0,1,1),\ (1,1,1,0),\ (0,1,1,1),\ (1,1,1,1).
\endaligned
$$

Now we consider the category $\mathscr{O}({\bf G})_\theta$ when $|I(\theta)|=2$. Assume $I(\theta)=\{r,s\}$. By the equivalence of  $\mathscr{O}({\bf G})_\theta$  and the right $\mathscr{A}_2$-modules, we know that except the four irreducible modules $E(\theta)_{\emptyset}, E(\theta)_{r},E(\theta)_{s}, E(\theta)_{\{r,s\}}$, the indecomposable projective modules $\mathbb{M}(\theta), \Delta(\theta)_{r}, \Delta(\theta)_{s}, \Delta(\theta)_{\{r,s\}}= E(\theta)_{\{r,s\}}$ and the indecomposable injective modules $\nabla(\theta)_{\{r, s\}}=E(\theta)_{\emptyset}, \nabla(\theta)_{r}, \nabla(\theta)_{s}, \mathbb{M}(\theta)$, the remaining indecomposable modules in $\mathscr{O}({\bf G})_\theta$ is $\mathbb{M}(\theta)/\Delta(\theta)_{\{r,s\}}$ and $\Delta(\theta)_{r}+\Delta(\theta)_{s}$.
Therefore all the indecomposable modules of $\mathscr{O}({\bf G})$ are known when the rank of ${\bf G}$ is 2.

\medskip

 Before we consider the algebra $\mathscr{A}_n$ for $n\geq 3$, we recall some facts and results about the Tits form of an algebra. For the algebra $\mathscr{A}\cong \Bbbk \mathcal{Q} /\mathcal{I}$, let $R$ be the minimal set of relations which generate the ideal $\mathcal{I}$. Then the Tits form of $\mathscr{A}$ is the integral quadratic
form $q_{\mathscr{A}}: \mathbb{Z}^{m} \rightarrow \mathbb{Z}$ defined by the formula
$$q_{\mathscr{A}}(x)=\sum_{i\in Q_0} x^2_i- \sum_{\alpha\in Q_1} x_{s(\alpha)}x_{t(\alpha)}+ \sum_{i,j\in Q_0}r_{ij}x_i x_j,$$
where $r_{ij}$ is the cardinality of $R\cap \Bbbk Q(i, j)$ and $\Bbbk Q(i, j)$ is the vector space
spanned by the paths from $i$ to $j$ (see \cite[Page 464]{B}). It is well know that if $\mathscr{A}$ is a tame algebra, the Tits form $q_{\mathscr{A}}$ is weakly positive, that is $q_{\mathscr{A}}(x)\geq 0$ for any $x\in \mathbb{Z}^{m}$ with nonnegative coordinates (see \cite[Section 1.3]{P}).

In our setting, the minimal set of relations $R$ which generate the ideal $\mathcal{I}$ contains all the relations
$$\alpha_{Y,U}\alpha_{U,Z}=\alpha_{Y,V}\alpha_{V,Z}$$ with $\alpha_{Y,U},\alpha_{U,Z},\alpha_{Y,V}, \alpha_{V,Z}$ are arrows in $Q_1$. Thus for $i_Y, i_Z \in  Q_0$, we have
$$r_{i_Y ,i_Z}=\left\{
\begin{array}{ll}
1 &\ \mbox{if}~Y\subset Z~\mbox{with}~ |Z\backslash Y|=2 ,\\
0 &\ \mbox{otherwise.}
\end{array}\right.$$
Therefore the Tits form of $\mathscr{A}_n$ is
$$q_n(x)=\sum_{Y} x^2_Y- \sum_{Y\subset Z,\ |Z\backslash Y|=1} x_{Y}x_{Z}+ \sum_{Y\subset Z, \ |Z\backslash Y|=2 }x_Y x_Z,$$
where $Y,Z$ are subsets of a fixed set $X$ such that $ |X|=n$.

\medskip
We consider the Tits form of the incidence algebra $\mathscr{A}_3$ which is  given by  the following quiver
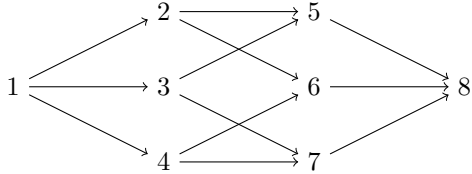
\begin{center}
\begin{tikzpicture}[scale=1]
\draw  (4,0) node (I1) {$1$} +(2,1) node (I2) {$2$} +(2,0) node (I3) {$3$} + (2,-1) node (I4) {$4$}+(4,1) node (I5) {$5$}+(4,0) node (I6) {$6$}+(4,-1) node (I7) {$7$}+(6,0) node (I8) {$8$};
\draw[->] (I1)--(I2) node[pos=.5,above] {};
\draw[->] (I1)--(I3) node[pos=.5,above] {};
\draw[->] (I1)--(I4) node[pos=.5,above] {};
\draw[->] (I2)--(I5) node[pos=.5,above] {};
\draw[->] (I2)--(I6) node[pos=.5,above] {};
\draw[->] (I3)--(I5) node[pos=.5,above] {};
\draw[->] (I3)--(I7) node[pos=.5,above] {};
\draw[->] (I4)--(I6) node[pos=.5,above] {};
\draw[->] (I4)--(I7) node[pos=.5,above] {};
\draw[->] (I5)--(I8) node[pos=.5,above] {};
\draw[->] (I6)--(I8) node[pos=.5,above] {};
\draw[->] (I7)--(I8) node[pos=.5,above] {};
\end{tikzpicture}
\end{center}
bounded by six relations, where each one is given by  $\omega_1=\omega_2$ of paths having the same starting and ending vertices of a parallelogram like the case of $\mathscr{A}_2$.
Now the Tits form is given by
$$
\aligned  q(x)
 = &\  x^2_1+x^2_2+x^2_3+x^2_4+x^2_5+x^2_6+x^2_7+x^2_8+x_1x_5+x_1x_6+x_1x_7\\
&\ +x_2x_8+x_3x_8+x_4x_8-x_1x_2-x_1x_3-x_1x_4-x_2x_5-x_2x_6 \\
&\ -x_3x_5-x_3x_7-x_4x_6-x_4x_7-x_5x_8-x_6x_8-x_7x_8.
\endaligned
$$
Do variable substitution with $x_1=x+z, x_8=z, y_1=x_2-x_5, y_2=x_2-x_6, y_3=x_3-x_5, y_4=x_3-x_7, y_5=x_4-x_6,y_6=x_4-x_7$, then we have
$$
\aligned  q(x)
 = &\  x^2+2xz+2z^2-\frac{1}{2}x (y_1+y_2+y_3+y_4+y_5+y_6)\\
&\ + \frac{1}{2}(y^2_1+y^2_2+y^2_3+y^2_4+y^2_5+y^2_6)
\endaligned
$$
which is nonnegative for any $x,z,y_i\in \mathbb{Z}$.  Thus the Tits form of $\mathscr{A}_3$ is  weakly positive.

In the following we use a basic method to show that $\mathscr{A}_3$ is a tame algebra (I do not know whether there is a criterion to get this property by the weakly positive Tits form).
We classify any indecomposable representation $\mathbb{V}=(V_i, f_{\alpha})$ of $\mathscr{A}_3$ to the following three cases by considering the morphism $f_{18}:V_1\rightarrow V_8$:
(a) When $V_1=V_8=0$, then $\mathbb{V}$ can be regarded as a representation of the Euclidean quiver of type $\widehat{A}_5$. (b) When $V_1\ne 0$ and $V_8\ne 0$, in this case $\text{dim}V_1=\text{dim}V_8=1$ and the only indecomposable module is the projective module $P(1)$. (c) When  $V_1=0$ or $V_8=0$, these two situations are symmetric and we consider the case $V_1=0$ and $V_8\ne 0$. Therefore the linear morphisms $f_{25}$,$f_{26}$,$f_{35}$,$f_{37}$,$f_{46}$,$f_{47}$,$f_{58}$,$f_{68}$,$f_{78}$ are all injective.
Since $\mathbb{V}$ is indecomposable, we also have $\text{dim}V_i \leq 1$ for $2\leq i \leq 4$ and $\text{dim}V_j \leq 2$ for $2\leq j \leq 4$ . Therefore there are finitely many indecomposable representations $\mathbb{V}=(V_i, f_{\alpha})$ such that $V_1=0$ and $V_8\ne 0$.

\medskip
Now consider the  incidence algebra $\mathscr{A}_4$. For the vertex $i_Y\in Q_0$,
we set
$$x_{Y}=\left\{
\begin{array}{ll}
0 &\ \mbox{when}~|Y|=0~\mbox{or}~4 ,\\
1 &\ \mbox{when}~|Y|=1~\mbox{or}~3 ,\\
2 &\ \mbox{when}~|Y|=2,
\end{array}\right.$$
for any vertex $i_Y\in Q_0$ in the Tits form of $\mathscr{A}_4$. Then by a direct calculation we have $q(x)=-4$. This example also implies that the Tits form of $\mathscr{A}_n$ is not weakly positive when $n\geq 4$. Then we have the following proposition.
\begin{Prop}\label{wild} The  algebra $\mathscr{A}_n$ is of wild type when $n\geq 4$.
\end{Prop}


\bibliographystyle{amsplain}

\begin{thebibliography}{10}


\bibitem {ASS}
Assem I, Skowronski A, Simson D. \textit{Elements of the Representation Theory of Associative Algebras: Volume 1: Techniques of Representation Theory}, Cambridge University Press, 2006.


\bibitem {ARS}
Auslander M, Reiten I, Smalo S O.  \textit{Representation theory of Artin algebras[M]}. Cambridge university press, 1997.

\bibitem {B}
Bongartz K. \textit{Algebras and Quadratic Forms}, Journal of the London Mathematical Society, 1983,s 2-28(3): 218-224.


\bibitem {BT}
Borel A, Tits J. \textit{Homomorphismes "Abstraits" de Groupes Algebriques Simples}, Annals of Mathematics, 1973, 97(3):499-571.


\bibitem {Ca}
R. W. Carter. \textit{Finite Groups of Lie Type: Conjugacy Classes and Complex Characters}, John Wiley and Sons., New York, 1985.


\bibitem {CD1}
Xiaoyu Chen, Junbin Dong. \textit{The Permutation Module on Flag Varieties in Cross Characteristic}, Math. Z. {\bf 293} (2019): 475-484.

\bibitem {CD2}
Xiaoyu Chen, Junbin Dong. \textit{The Decomposition of Permutation Module for Infinite Chevalley Groups}, Science China Mathematics, to appear.

\bibitem{CD3}
Xiaoyu Chen, Junbin Dong.  \textit{Abstract Induced Modules for  Reductive Algebraic Groups with Frobenius Maps}, Int. Math. Res. Not.(in press).

\bibitem {CPS}
Cline E, Parshall B, Scott L. \textit{Finite dimensional algebras and highest weight categories}, Journal Fur Die Reine Und Angewandte Mathematik, 1988(391):85-99.


\bibitem{D}
Junbin Dong. \textit{Irreducibility of Certain Subquotients of Spherical Principal Series Representations of Reductive Groups with Frobenius Maps}, arXiv: 1702.01888v2.



\bibitem {FS}
Farkas, D. R. , Snider, R. L.  \textit{Group algebras whose simple modules are injective}, Transactions of the American Mathematical Society, 194(JUL), 241-248.



\bibitem {P}
de la Pena J A.   \textit{On the representation type of one point extensions of tame concealed algebras}, manuscripta mathematica volume 61 (1988), 183-194.




\bibitem {X}
Nanhua Xi. \textit{Some Infinite Dimensional Representations of Reductive Groups With Frobenius Maps}, Science China Mathematics \textbf{57} (2014), 1109--1120.




\end{thebibliography}

\end{document}